\long\def\symbolfootnote[#1]#2{\begingroup\def\thefootnote{\fnsymbol{footnote}}\footnote[#1]{#2}\endgroup}

\documentclass{amsart}

\usepackage{hyperref}
\usepackage{indentfirst}
\usepackage{amssymb}
\usepackage{amsmath}
\usepackage{amsfonts}

\newtheorem{theorem}{Theorem}[section]

\newtheorem{lemma}[theorem]{Lemma}
\theoremstyle{remark}

\theoremstyle{definition}
\newtheorem{definition}[theorem]{Definition}

\theoremstyle{proposition}
\newtheorem{proposition}[theorem]{Proposition}
\newtheorem{conjecture}[theorem]{Conjecture}
\numberwithin{equation}{section}

\begin{document}
\author{Mengqing Zhan, Linfeng Zhou}
\title[]{The isoperimetric problem in the 2-dimensional Finsler space forms with $k=0$. \uppercase\expandafter{\romannumeral2}}
\date{}
\maketitle

\begin{abstract} 
This paper is a continuation of the second author's previous work\cite{Zh2}. We investigate the isoperimetric problem in the 2-dimensional Finsler space form $(F_B, B^2(1))$ with $k=0$ by using the Holmes-Thompson area and prove that the circle centered the origin achieves the local maximum area of the isoperimetric problem.\\

\end{abstract}

\section{\textbf{Introduction}}

In this paper, we define the \emph{Finsler space forms} to be the spherically symmetric projectively flat Finsler metrics of constant flag curvatures.  Actually, the Finsler space forms can be completely classified as follows (see \cite{Zh1}, \cite{MZ}, \cite{L}):

\begin{theorem}Up to a scaling, a spherically symmetric projectively flat Finsler metric $(\Omega, F)$ has a constant flag curvature $k$ if and only if 
\begin{enumerate}
\item[(1)] $k=1$, the metric is either the projective sphere model or Bryant metric \[F=\frac{|y|c_1(z_1)}{c_1(z_1)^2+\big(z_2+c_2(z_1)\big)^2}\] where
$z_1:=\sqrt{|x|^2-\frac{\langle x,y\rangle^2}{|y|^2}}, z_2:=\frac{\langle x,y\rangle}{|y|},
c_1(z_1):=\frac{\sqrt{2}}{2}\sqrt{2d_2+z_1^2+\sqrt{(2d_2+z_1^2)^2+4d_1^2}},$
$$c_2(z_1):=\pm\frac{\sqrt{2}}{2}\sqrt{-2d_2-z_1^2+\sqrt{(2d_2+z_1^2)^2+4d_1^2}},$$
$d_1$ and $d_2$ are positive real numbers, $\Omega=R^n$.

\item[(2)] $k=0$, the metric is either the Euclidean space $R^n$ or the Berwald's example where 
\[F=\frac{(\sqrt{|y|^2-(|x|^2|y|^2-\langle x,y\rangle^2)}+\langle x,y\rangle)^2}{(1-|x|^2)^2\sqrt{|y|^2-(|x|^2|y|^2-\langle x,y\rangle^2)}}, \Omega=B^n(1).\]

\item[(3)] $k=-1$, the metric is either the Klein metric (one model of the hyperbolic space)
\[F=\frac{\sqrt{|y|^2-(|x|^2|y|^2-\langle x,y\rangle^2)}}{1-|x|^2}, \Omega=B^n(1); \]
or the Funk metric of Randers type
\[F=\frac{\sqrt{|y|^2-(|x|^2|y|^2-\langle x,y\rangle^2)}+\langle x,y\rangle}{2(1-|x|^2)}, \Omega=B^n(1); \]
or $\Omega=B^n(\sqrt{2(d_2-d_1)})$ with Finsler metric
\[F=\frac{|y|c_1(z_1)}{c_1(z_1)^2-\big(z_2+c_2(z_1)\big)^2}\] where
$z_1:=\sqrt{|x|^2-\frac{\langle x,y\rangle^2}{|y|^2}}, z_2:=\frac{\langle x,y\rangle}{|y|},
c_1(z_1):=\frac{\sqrt{2}}{2}\sqrt{2d_2-z_1^2+\sqrt{(2d_2-z_1^2)^2-4d_1^2}},$
\[c_2(z_1):=\pm\frac{\sqrt{2}}{2}\sqrt{2d_2-z_1^2-\sqrt{(2d_2-z_1^2)^2-4d_1^2}},\] and $d_2>d_1$ are positive real numbers.
\end{enumerate}
\end{theorem}
\noindent Since the classical Riemannian space forms are included in the above theorem, the Finsler space forms we defined here can be viewed as a proper generalization of the Riemannian space forms. 

The famous 2-dimensional isoperimetric problem is that in a certain space, how to maximize the area of a simple closed curve under the constraint that the length of the curve to be fixed?

In mathematical history, this problem is quite old and many mathematicians have done lots of work on it. It can be generalized to the high dimensional compact Riemannian manifolds.  What's more, the isoperimetric inequality is proved to be equivalent to the Sobolev inequality which is often used in PDE theory. One can see a good survey article, written by  R. Osserman \cite{O}, about the developments in the theory of the isoperimetric problem.  For the Finslerian case, one can see \cite{PT} and \cite{Zh2} for the details. 

In the previous paper \cite{Zh2}, the second author proves that in 2-dimensional Finsler space $(F_B, B^2(1))$ with $k=0$ by using the Busemann-Hausdorff area,  the circle centered the origin achieves the local maximum area of the isoperimetric problem. 

Employing the similar method, we prove that \emph{in 2-dimensional Finsler space $(F_B, B^2(1))$ with $k=0$ by using the Holmes-Thompson area,  the circle centered the origin also achieves the local maximum area of the isoperimetric problem.} 

\section{\textbf{A brief review of the sufficiency theorem for isoperimetric problem in the calculus of variations}}
 
In the classical calculus of variations, the isoperimetric problem to be considered is that of minimizing an integral 
$$I=\int_{a}^{b}f(t,x(t),\dot{x}(t))dt=\int_{a}^{b}f(t,x_1,\dots,x_n,\dot{x}_1,\dots,\dot{x}_n)dt$$
in a class of admissible arcs
$$x_i(t), a\leq t\leq b; i=1,\dots,n,$$
joining two fixed points and satisfying a set of isoperimetric conditions 
$$I_{\alpha}=\int_{a}^{b}f_{\alpha}(t,x(t),\dot{x}(t))dt=l_{\alpha}, \alpha=1,\dots,m,$$
where the $l's$ are constants. 
It is assumed that the functions $f, f_{\alpha}$ are defined and have continuous derivatives of the first three orders in a region $R$ of points $(t,x,\dot{x})$. The points of $R$ will be called admissible. A continuous arc that can be divided into a finite number of suburbs on each of which it has continuous derivatives will be called admissible if its elements $(t,x,\dot{x})$ are all admissible. 

Before stating the sufficiency theorem proved in \cite{He}, let us recall some concepts.

An admissible arc $x_0$ is called a strong minimum of above isoperimetric problem  if there is a neighborhood $\mathcal{F}$ of $x_0$ in $tx$-space such that the inequality $I(x)\geq I(x_0)$ holds for all admissible arcs $x\neq x_0$ whose elements $(t,x(t))$ are in $\mathcal {F}$. If the inequality can be replaced by a strict inequality, the minimum is said to be a proper strong minimum.

The sufficiency conditions are given in terms of an integral $J(x)$ of the form
$$J(x)=I+\sum_{\alpha=1}^{m}\lambda_{\alpha}I_{\alpha}=\int_{a}^{b}F(t,x(t),\dot{x}(t),\lambda)dt,$$
where $F=f+\sum_{\alpha=1}^{m}\lambda_{\alpha}f_{\alpha}$ and the $\lambda$'s are constant multipliers.  An admissible arc and a set of constants $\lambda_{\alpha}$ having continuous second derivatives will be said to form an isoperimetric extremal if they satisfy the Euler-Lagrange equations 
\begin{equation}
F_{x_i}-\frac{dF_{\dot{x}_i}}{dt}=0.
\end{equation}

Let an admissible $x_0$ be an isoperimetric extremal joining the two points and satisfying the condition (2). The arc $x_0$ is said to be normal, if the equation $P_{i\alpha}a_{\alpha}=0$ where 
$$P_{i\alpha}=(f_{\alpha})_{ x_i}-\frac{d (f_{\alpha})_{\dot{x}_i}}{dt},$$
hold along $x_0$ only in case the constants $a_{\alpha}$ are all zero. 

The Weierstrass E-function for $F$ is given by
$$E(t,x,\dot{x},u):=F(t,x,u)-F(t,x,\dot{x})-\sum_{i=1}^{n}(u_i-\dot{x}_i)F_{\dot{x}_i}(t,x,\dot{x}).$$
An arc $x_0$ is said to satisfy the strict Weierstrass condition if for each $(t,x,\dot{x})$ in a neighborhood of $x_0$, the Weierstrass function
$$E(t,x,\dot{x},u)>0$$ 
holds for every admissible set $(t,x,u)\neq (t,x,\dot{x})$.

Finally the second variation of $J^{\prime\prime}(x_0,y)$ of $J$ along $x_0$ is of the form 
$$J^{\prime\prime}(x_0,y)=\int_a^b2\omega(t,y(t),\dot{y}(t))dt,$$
where $2\omega=\sum_{i,j}F_{x_ix_j}y_iy_j+2F_{x_i\dot{x}_j}y_i\dot{y}_j+F_{\dot{x}_i\dot{x}_j}\dot{y}_i\dot{y}_j$.

A sufficiency theorem for a strong minimum of the isoperimetric problem, which is proved by Hestenes in \cite{He}, can be stated as the following:
 \begin{theorem}
 Let $x_0$ be an admissible arc. Suppose there exist $\lambda_1,\dots,\lambda_m$ such that, relative to the function
$$J(x)=I+\sum_{\alpha=1}^{m}\lambda_{\alpha}I_{\alpha}=\int_{a}^{b}F(t,x(t),\dot{x}(t),\lambda)dt,$$
\begin{enumerate}
\item $x_0$ is isoperimetric extremal,
\item $x_0$ is normal,
\item $x_0$ satisfies the strict Weierstrass condition,
\item $J^{\prime\prime}(x_0,y)>0$ for every non-null admissible variations $y_i(t), (a\leq t\leq b)$, vanishing at $t=a$ and $t=b$ and satisfying with $x_0$ the equations 
         $$\int_{a}^b\big( (f_{\alpha})_{ x_i}y_i+(f_{\alpha})_{ \dot{x}_i}\dot{y}_i \big)dt=0,$$
\item along $x_0$ the inequality $\sum_{i,j}F_{\dot{x}_i\dot{x}_j}y_iy_j>0$ holds for every set $(y)\neq (0).$        
\end{enumerate}
Then $x_0$ is a proper strong minimum of the isoperimetric problem. 
 \end{theorem}
 
 A point $t=c$ on $a<t\leq b$ will be said to define a conjugate point to $t=a$ on $x_0$ if there is a solution $y_i(t)$ of the Jacobi equations
 $$\omega_{y_i}-\frac{d\omega_{\dot{y}_i}}{dt}=0$$
such that 
$$\int_{a}^b\big( (f_{\alpha})_{ x_i}y_i+(f_{\alpha})_{ \dot{x}_i}\dot{y}_i \big)dt=0.$$
Furthermore, the solution $y(t)$  satisfies $y(a)=y(c)=0$ and $y(t)\not\equiv 0$ on $a<t<c$.

\begin{theorem}
The inequality $J^{\prime\prime}(x_0,y)>0$ holds for all $y\neq 0$ if and only if there is no point $c$ conjugate to $a$ on $0<t\neq b$ along $x_0$.
\end{theorem}
  
For our purpose, we shall concern the calculus of variations in parametric form in 2-dimensional Euclidean space.  Let an arc be of the parametric form
  $$x(t):= (x_1(t), x_2(t))  \quad (a\leq t\leq b)$$
in $R^2$, where the functions $x_i(t)$ are assumed to be continuous first and second derivatives $\dot{x}_i(t)$ for $i=1, 2$ and if more over $\dot{x}_1^2+\dot{x}_2^2\neq 0$ in $(a,b)$. 

An integral 
$$\int_a^b f(t,x(t),\dot{x}(t))dt$$
can be shown to be independent of the parametrization of the arc $x$ if and only if the integrand $f$ is independent of $t$ and is positively homogeneous in $\dot{x}$ of degree one. 

The isoperimetric problem in parametric form can be stated as the following type.
Among all admissible curves joining two given points for which the integral
$$K=\int_a^b g(x_1,x_2,\dot{x}_1,\dot{x}_2)dt$$
takes a given value $l$, how to determine the one which minimizes the integral 
  $$I=\int_a^b f(x(t),\dot{x}(t))dt=\int_a^b f(x_1,x_2,\dot{x}_1,\dot{x}_2)dt.$$ 
The two function $f$ and $g$ are positively homogeneous in $\dot{x}$ of degree one and have continuous derivatives of the first three orders.

Let 
$$J(x)=I+\lambda K=\int_a^b H(x_1,x_2,\dot{x}_1,\dot{x}_2,\lambda)dt$$
where $H=f+\lambda g$. With a slight difference, the sufficient theorem proved by Hestenes can be illustrated by the following.
\begin{theorem} Let $x_0$ be an admissible arc. Suppose there exist $\lambda_0$ such that, relative to the function
$$J(x)=\int_{a}^{b}H(x_1,x_2,\dot{x}_1,\dot{x}_2,\lambda)dt,$$
\begin{enumerate}
\item $x_0$ is isoperimetric extremal, i.e.
  $$H_{x_i}-\frac{dH_{\dot{x}_i}}{dt}=0$$
  for $i=1,2$,
\item $x_0$ is normal which means the equation $P_{i}\neq 0$ where 
$$P_{i}=g_{ x_i}-\frac{d g_{\dot{x}_i}}{dt},$$
hold along $x_0$ for $i=1,2$,

\item  the Weierstrass function
$$E(x,\dot{x},u)>0$$ 
holds for every admissible set $(x,u)$ with  $u\neq k\dot{x} (k>0)$,

\item $J^{\prime\prime}(x_0,y)>0$ for every admissible variations $y_i(t)\neq \rho(t)\dot{x}_{0i}(t), (a\leq t\leq b)$, vanishing at $t=a$ and $t=b$ and satisfying with $x_0$ the equations 
         $$\int_{a}^b\big( g_{ x_i}y_i+g_{ \dot{x}_i}\dot{y}_i \big)dt=0\quad (i=1,2),$$
\item along $x_0$ the inequality $\sum_{i,j=1}^2H_{\dot{x}_i\dot{x}_j}y_iy_j>0$ holds for all $y\neq k\dot{x}_0(t).$        
\end{enumerate}
Then $x_0$ is a proper strong minimum of the isoperimetric problem. 
\end{theorem}

In the Bolza's book \cite{Bo}, the Jacobi equation along an isoperimetric extremal admissible curve $x_0$ turns out to be  
\[\Psi(y)+\mu U=0\]
where $\Psi(y)=H_2y-\frac{d}{dt}(H_1\dot{y})$,   $H_1=\frac{H_{\dot{x}_1\dot{x}_1}}{\dot{x}_2^2}$, 
$H_2=\frac{H_{x_1x_1}-\ddot{x}_2^2H_1-\frac{d L}{dt}}{\dot{x}_2^2}$, $L=H_{x_1\dot{x}_1}-\dot{x}_2\ddot{x}_2H_1$, $U=g_{x_1\dot{x}_2}-g_{\dot{x}_1x_2}+g_1(\dot{x}_1\ddot{x}_2-\ddot{x}_1\dot{x}_2)$, $g_1=\frac{g_{\dot{x}_1\dot{x}_1}}{\dot{x}_2^2}$. Moreover, the solution $y(t)$ shall satisfy the condition 
\[\int_a^b Uy dt=0.\]
 A conjugate point  $t=c$ to the point $t=a$ on $x_0$ means there is a solution $y(t)$ of the above Jacobi equation with the integral condition satisfying $y(a)=y(c)=0$ and $y(t)\not\equiv 0$ on $a<t<c$.
 
With a same method in the previous theorem, it can be proved that  $J^{\prime\prime}(x_0,y)>0$ holds if and only if there is no point $c$ conjugate to $a$ on $0<t\neq b$.

\section{\textbf{The Holmes-Thompson area in the 2-dimensional Finsler space forms with $k=0$}}
From the previous definition, the 2-dimensional  non-Riemannian Finsler space form with $k=0$ can be written as \begin{equation}\label{eqn1}
F_B=|y|\phi\left(|x|,\frac{\langle x,y\rangle}{|y|}\right)=\frac{\left(\sqrt{|y|^2-|x|^2|y|^2+ \langle x,y\rangle^2} + \langle x,y \rangle \right)^2}{(1-|x|^2)^2 \sqrt{|y|^2-|x|^2|y|^2+ \langle x,y\rangle^2} },
\end{equation}
which defined on a ball $B^2(1)=\{x\in R^2| |x|<1\}$.  
If let $|x|=r, \frac{\langle x,y\rangle}{|y|}=s$, then by the formula \ref{eqn1}, we can get\begin{equation}\label{eqn2}
\phi(r,s)=\frac{(\sqrt{1-r^2+s^2}+s)^2}{(1-r^2)^2\sqrt{1-r^2+s^2}}
\end{equation}  
by the notation of the spherically symmetric Finsler metrics. 

The Holmes-Thompson volume from $dV_{HT}=\sigma_{HT}(x)dx$ of a Finsler metric $F$ is defined as following:
\[\sigma_{HT}(x)=\frac{1}{Vol(B^n)}\int_{F(x,y)<1}det(g_{ij}(x,y))dy,\]
where $Vol(B^n)$ means the Euclidean volume of a unit ball in $R^n$.
To get the Holmes-Thompson area element $\sigma_{HT}(r)$, we need the following formula which is the similar to the $(\alpha, \beta)$ metric \cite{CS}. 
\begin{lemma}
Let $F=u\phi(r,s)$ be a spherically symmetric Finsler metric on $\Omega$. Its Holmes-Thompson volume element $\sigma_{HT}$ of $F$ is given by
\begin{equation}\label{eqn3}
\sigma_{HT}(r)=\frac{\int_0^\pi(\sin^{n-2} t)T(r, r\cos t)\mathrm{d}t}{\int_0^\pi \sin^{n-2} t\mathrm{d}t}, 
\end{equation}
where $T(r,s):=\phi(\phi-s\phi_s)^{n-2}[(\phi-s\phi_s)+(r^2-s^2)\phi_{ss}]$.
\end{lemma}  

\begin{theorem}
For 2-dimensional non-Riemannian Finsler space form with k=0 $(F_B, B^2(1))$, the Holmes-Thompson area element is given by
\[\sigma_{HT}(r)=\frac{1+r^2-\frac{1}{8}r^4}{(1-r^2)^{\frac{7}{2}}},\]
where $r:=|x|$.
\end{theorem}
\begin{proof}
Substituting $\phi(r,s)=\frac{(\sqrt{1-r^2+s^2}+s)^2}{(1-r^2)^2\sqrt{1-r^2+s^2}}$ into $T(r,s)$ will yield 
\[T(r,s)=\frac{3-2\rho^2}{\rho^6(\rho-s)^2},\]
where $\rho=\sqrt{1-r^2+s^2}$.
So the Holmes-Thompson area element is
\begin{equation}\label{eqn5}\sigma_{HT}(r)=\frac{1}{\pi}\int_0^\pi \frac{1+2r^2 \sin^2 t}{(1-r^2\sin^2 t)^3(\sqrt{1-r^2\sin^2 t} -r\cos t)^2} \mathrm{d}t=\frac{1+r^2-\frac{1}{8}r^4}{(1-r^2)^{\frac{7}{2}}}.\end{equation}  
\end{proof}

Furthermore, one can reduce the Holmes-Thompson area to the line integral.
\begin{theorem}
Suppose a domain $\Omega$ is enclosed by a simple closed curve $c(t)=(x_1(t), x_2(t))$ where $t\in [t_0, t_1]$ and $c(t_0)=c(t_1)$. The Holmes-Thompson area $A_{HT}$ of the domain $\Omega$ in $(F_B, B^2(1))$ is of form
\[ A_{HT}=\frac{1}{8} \int_{t_0}^{t_1} \frac{x_1^2+x_2^2-4}{(1-x_1^2-x_2^2)^{\frac{5}{2}}} (x_2\dot{x_1}-x_1\dot{x_2}) \mathrm{d}t .\]
\end{theorem}
\begin{proof}
From above theorem, the Holmes-Thompson area enclosed by $c(t)$ is given by 
\[A_{HT}=\iint_{\Omega} \frac{1+(x_1^2+x_2^2)-\frac{1}{8}(x_1^2+x_2^2)^2}{(1-x_1^2-x_2^2)^{\frac{7}{2}}} \mathrm{d}x_1\mathrm{d}x_2.\]
Let $P(x_1,x_2)=\frac{x_2(x_1^2+x_2^2-4)}{8(1-x_1^2-x_2^2)^{\frac{5}{2}}}$ and $Q(x_1,x_2)=-\frac{x_1(x_1^2+x_2^2-4)}{8(1-x_1^2-x_2^2)^{\frac{5}{2}}}$, then
\[\frac{\partial Q}{\partial x_1}-\frac{\partial P}{\partial x_2}=\frac{1+(x_1^2+x_2^2)-\frac{1}{8}(x_1^2+x_2^2)^2}{(1-x_1^2-x_2^2)^{\frac{7}{2}}}.\]
Thus the Green formula implies the result.  
\end{proof}

\section{\textbf{The Euler-Langrange equation of the isoperimetric problem in the 2-dimensional Finsler space forms with k = 0}}
Now we consider an abitrary smooth closed curve $c(t)=(x_1(t),x_2(t)), t\in[t_0,t_1], c(t_0)=c(t_1)$. While its circumference $L=\int_{t_0}^{t_1} g(x_1,x_2,\dot{x_1},\dot{x_2}) dt$ is given, we want to figure out when its area $A_{HT}=\int_{t_0}^{t_1} f(x_1,x_2,\dot{x_1},\dot{x_2}) \mathrm{d}t$ will reach its maximum, and what the maximum curve $c(t)$ is. For obtaining the circumference and area, we use the following integrand \begin{equation}\label{eqn14} g(x_1,x_2,\dot{x_1},\dot{x_2}) =\frac{(1-x_1^2-x_2^2)(\dot{x_1}^2+\dot{x_2}^2)+2(x_1\dot{x_1}+x_2\dot{x_2})^2}{(1-x_1^2-x_2^2)^2\sqrt{(1-x_1^2-x_2^2)(\dot{x_1}^2+\dot{x_2}^2)+(x_1\dot{x_1}+x_2\dot{x_2})^2}}, \end{equation}
\begin{equation}\label{eqn15}  f(x_1,x_2,\dot{x_1},\dot{x_2})=\frac{x_1^2+x_2^2-4}{8(1-x_1^2-x_2^2)^{\frac{5}{2}}} (x_2\dot{x_1}-x_1\dot{x_2}). \end{equation}
Resorting to Lagrange multiplier method, let\begin{equation}\label{eqn16}  J=A_{HT}+\lambda L=\int_{t_0}^{t_1} f(x_1,x_2,\dot{x_1},\dot{x_2})+\lambda g(x_1,x_2,\dot{x_1},\dot{x_2}) \mathrm{d}t=\int_{t_0}^{t_1} h (x_1,x_2,\dot{x_1},\dot{x_2},\lambda) \mathrm{d}t.\end{equation}
The Euler-Lagrange equations of $J$ are \begin{equation}\label{eqn17}  \begin{cases} \frac{\partial h}{\partial x_1}- \frac{\mathrm{d}}{\mathrm{d} t} \frac{\partial h}{\partial \dot{x_1}} =0 \\ \frac{\partial h}{\partial x_2}- \frac{\mathrm{d}}{\mathrm{d} t} \frac{\partial h}{\partial \dot{x_2}} =0 .\end{cases}\end{equation}

Again, we adopt the polar coordinate to reduce the equation$\gamma(t)=(r(t) \cos t, r(t) \sin t)$, 
š\begin{equation}\label{eqn18}  \begin{cases} \dot{x_1}=\dot{r}\cos t-r\sin t \\ \dot{x_2}=\dot{r}\sin t+r\cos t .\end{cases}\end{equation}
A further calculation reveals \begin{equation}\label{eqn19}  \begin{cases} x_1^2+x_2^2=r^2 \\ \dot{x_1}^2+\dot{x_2}^2=r^2+\dot{r}^2 \\ x_1\dot{x_1}+x_2\dot{x_2}=r\dot{r} \\ x_2\dot{x_1}-x_1\dot{x_2}=-r^2. \end{cases}\end{equation}

Substituting these into (\ref{eqn14}), (\ref{eqn15}) will get:\begin{equation}\label{eqn20} g(r,\dot{r})= \frac{r^2(1-r^2)+(1+r^2)\dot{r}^2}{(1-r^2)^2\sqrt{r^2(1-r^2)+\dot{r}^2}} ,\end{equation}\begin{equation}\label{eqn21} f(r)= \frac{r^2(4-r^2)}{8(1-r^2)^{\frac{5}{2}}}. \end{equation}

Therefore we have
\begin{equation}\label{eqn22}\begin{cases} \frac{\partial h}{\partial r} = \frac{r(1+r^2-\frac{1}{8}r^4)}{(1-r^2)^{\frac{7}{2}}} -\lambda \frac{r\left( r^2-3r^6+2r^8+2\dot{r}^4(3+r^2)+\dot{r}^2(1+6r^2-3r^4-4r^6) \right)}{(r^2-1)^3(\dot{r}^2+r^2-r^4)^{\frac{3}{2}}} \\

\frac{\partial h}{\partial \dot{r}}=\lambda \frac{\dot{r}\left( r^2+r^4-2r^6+\dot{r}^2(1+r^2) \right)}{(r^2-1)^2(\dot{r}^2+r^2-r^4)^{\frac{3}{2}}}.
\end{cases}  \end{equation}

Plugging above calculations into the Euler-Lagrange equation $\frac{\partial h}{\partial r}- \frac{\mathrm{d}}{\mathrm{d} t} \frac{\partial h}{\partial \dot{r}} =0$ and integrating the equation,  we arrive at the following proposition:
\begin{proposition}
For a curve $\gamma(t)=(r(t)\cos t,r(t)\sin t)$ which is a solution of the isoperimetric problem in $(F_B,B^2(1))$ when using the Holmes-Thompson area, then it must satisfy the equation:
\begin{equation}\label{eqn23}
\frac{r^2(4-r^2)}{8(1-r^2)^{\frac{5}{2}}}  + \lambda \frac{r^2(\dot{r}^2+r^2)}{(\dot{r}^2+r^2-r^4)^{\frac{3}{2}}}=C,
\end{equation}
where $C$ is a constant. 
\end{proposition}

It is easy to see that $r=constant$ is one of the solutions. 

\begin{theorem}
The circles $c_0=(a\cos(t), a\sin(t))$ centered at the origin, where $a\in (0,1)$ and $t\in [0, 2\pi]$, are isoperimetric extremal with respect to the integral $J$. Moreover, 
$$\lambda_{0}=-\frac{a(1+a^2-\frac{1}{8}a^4)}{1+a^2-2a^4}<0.$$
\end{theorem}
\begin{proof}
The curves $c_0=(a\cos(t), a\sin(t))$ are obviously the solutions of the Euler-Lagrange Equations of $J$. Plugging $c_0=(a\cos(t), a\sin(t))$ into the equation (\ref{eqn17}) will obtain $$\lambda_{0}=-\frac{a(1+a^2-\frac{1}{8}a^4)}{1+a^2-2a^4}.$$
For $a \in (0,1)$, $\lambda_{0}$ is negative.
\end{proof}

\begin{theorem}
The circles $c_0=(a\cos(t), a\sin(t))$ centered at the origin  are normal, where $a\in (0,1)$ and $t \in [0, 2\pi]$. 
\end{theorem}
\begin{proof}
Along the circles $c_0$, it can be calculated that
  \[P_1=g_{x_1}-\frac{d g_{\dot{x}_1}}{dt}=\frac{(1+2a^2)}{(1-a^2)^{\frac{5}{2}}}\cos(t),\]
  \[P_2=g_{x_2}-\frac{d g_{\dot{x}_2}}{dt}=\frac{(1+2a^2)}{(1-a^2)^{\frac{5}{2}}}\sin(t).\]
Therefore $P_1$ and $P_2$ are not $0$ function and the circles $c_0$ are normal.  
\end{proof}

\section{\textbf{Weierstrass E-function}}
\begin{definition}
The Weierstrass function of the integral $J$ is defined as
 \[E(x_1,x_2,\dot{x_1},\dot{x_2},p_1,p_2):=h(x_1,x_2,p_1,p_2)-h(x_1,x_2,\dot{x_1},\dot{x_2})\]
\[-(p_1-\dot{x_1})\frac{\partial h(x_1,x_2,\dot{x}_1,\dot{x}_2)}{\partial \dot{x}_1}-(p_2-\dot{x_2})\frac{\partial h(x_1,x_2,\dot{x}_1,\dot{x}_2)}{\partial \dot{x}_2}.\]

\end{definition}

\begin{proposition}
If assuming $\lambda<0$, then the Weierstrass E-function $E$ of $J$  satisfies: $$E<0$$
except for $(p_1,p_2)\neq k(\dot{x_1},\dot{x_2}), (k>0)$.
\begin{proof}
If let $$A(x_1,x_2)=\sqrt{1-x_1^2-x_2^2},$$ $$\alpha(x_1,x_2,\dot{x_1},\dot{x_2})=\sqrt{A^2(\dot{x_1}^2+\dot{x_2}^2)+(x_1\dot{x_1}+x_2\dot{x_2})^2},$$ and $\beta(x_1,x_2,\dot{x_1},\dot{x_2})=x_1\dot{x_1}+x_2\dot{x_2}$, thus
 \begin{equation}\label{eqn25} f(x_1,x_2,\dot{x_1},\dot{x_2})=\frac{A^2+3}{8A^5}(x_1\dot{x_2}-x_2\dot{x_1})
 \end{equation}
and \begin{equation}\label{eqn26} g(x_1,x_2,\dot{x_1},\dot{x_2})=\frac{\alpha^2+\beta^2}{A^4\alpha}.\end{equation}
Furthermore, one can calculate \begin{equation}\label{eqn27} \frac{\partial{\alpha}}{\partial \dot{x_1}}= \frac{1}{\alpha}(A^2\dot{x_1}+\beta x_1), \frac{\partial{\alpha}}{\partial \dot{x_2}}= \frac{1}{\alpha}(A^2\dot{x_2}+\beta x_2).\end{equation}
A bit of effort makes the following expression 
\begin{equation}\label{eqn28}  \begin{aligned} & \frac{\partial h}{\partial \dot{x_1}}  = \frac{\partial f}{\partial \dot{x_1}} +\lambda  \frac{\partial g}{\partial \dot{x_1}} =-x_2 \frac{A^2+3}{8A^5}  + \frac{ \lambda }{\alpha^3 A^4}\left(A^2\dot{x_1}( \alpha^2 - \beta^2) + \beta x_1 (3 \alpha^2-\beta^2) \right) ,\\ 
& \frac{\partial h}{\partial \dot{x_2}} = x_1\frac{A^2+3}{8A^5}  + \frac{ \lambda }{\alpha^3 A^4}\left(A^2\dot{x_2}( \alpha^2 - \beta^2) + \beta x_2 (3 \alpha^2-\beta^2) \right) .\end{aligned} \end{equation}
By taking the definition of E-function, we have \begin{equation}\label{eqn29} \begin{aligned}
E =& f(x_1,x_2,p_1,p_2) + \lambda g(x_1,x_2,p_1,p_2) - f(x_1,x_2,\dot{x_1},\dot{x_2}) - \lambda g(x_1,x_2,\dot{x_1},\dot{x_2})\\
&- (p_1-\dot{x_1}) ( \frac{\partial f}{\partial \dot{x_1}} +\lambda  \frac{\partial g}{\partial \dot{x_1}}) )- (p_2-\dot{x_2}) ( \frac{\partial f}{\partial \dot{x_2}} +\lambda  \frac{\partial g}{\partial \dot{x_2}}) \\
=&\frac{A^2+3}{8A^5}(x_1p_2-x_2p_1)+\lambda g(x_1,x_2,p_1,p_2) - \frac{A^2+3}{8A^5}(x_1\dot{x_2}-x_2\dot{x_1}) - \lambda g(x_1,x_2,\dot{x_1},\dot{x_2}) \\
& -(\dot{x_1}-p_1) x_2 \frac{A^2+3}{8A^5} +\lambda (\dot{x_1}-p_1)  \frac{\partial g}{\partial \dot{x_1}}  +(\dot{x_2}-p_2) x_1 \frac{A^2+3}{8A^5} + \lambda (\dot{x_2}-p_2)  \frac{\partial g}{\partial \dot{x_2}} \\
=& \lambda \left( g(x_1,x_2,p_1,p_2)-g(x_1,x_2,\dot{x_1},\dot{x_2})+(\dot{x_1}-p_1) \frac{\partial g}{\partial \dot{x_1}} + (\dot{x_2}-p_2) \frac{\partial g}{\partial \dot{x_2}} \right).
\end{aligned}\end{equation}
Since $E$ does not involve $f$,  the Weierstrass E-function $E$ is independent of the explicit area form. By following the steps in \cite{Zh2}, one will get the result.
\end{proof}
\end{proposition}

From the proposition, we have the following theorem.
 \begin{theorem}
 Let $c_0$ be the circle centered at the origin in $B^2(1)$. For each $(x, \dot{x})$ in a neighborhood of $c_0$, the Weierstrass function
 \[E(x,\dot{x},u)<0\]
 holds for every admissible set $(x,u)\neq (x, k\dot{x}) (k>0)$.
 \end{theorem}

\section{\textbf{The conjugate points of the critical circles}}
For the isoperimetric extremal circles $c_0=(a\cos t,a\sin t), a\in(0,1)$ in $(F_B,B^2(1))$, the integrands of the  Holmes-Thompson area $A_{HT}$ and the length $L$ are given by
$$f(c_0)=\frac{a^2(4-a^2)}{8(1-a^2)^{\frac{5}{2}}}, \quad g(c_0)=\frac{a}{(1-a^2)^{\frac{3}{2}}}.$$
Hence 
$$h(c_0)=f(c_0)+\lambda g(c_0)=\frac{a^2(4-a^2)}{8(1-a^2)^{\frac{5}{2}}}+\lambda\frac{a}{(1-a^2)^{\frac{3}{2}}}.$$
 The Jacobi equation along $c_0$ is 
 \begin{equation}\label{eqn30}	
\Psi(w)+\mu U=0,
\end{equation}
where\begin{equation}\label{eqn31}\begin{cases} \Psi(w)=H_2 w-\frac{\mathrm{d}}{\mathrm{d}t} (H_1 w') ,\\
U=g_{x_1\dot{x}_2}-g_{\dot{x}_1x_2}+g_1(\dot{x}_1\ddot{x}_2-\ddot{x}_1\dot{x}_2) ,\\
g_1=\frac{g_{\dot{x}_1\dot{x}_1}}{\dot{x}_2^2} ,\\
H_1=\frac{h_{\dot{x}_1\dot{x}_1}}{\dot{x}_2^2} ,\\
H_2=\frac{h_{x_1x_1}-\ddot{x}_2^2 H_1-\frac{\mathrm{d}J}{\mathrm{d} t}}{\dot{x}_2^2}, \\
J=h_{x_1\dot{x}_1}- \dot{x}_2\ddot{x}_2 H_1.
 \end{cases}\end{equation}
 Furthermore, $w$ satisfies 
\[\int_{t_0}^t U w\mathrm{d}t=0.\]
Plugging $g$, $h$ and $\lambda$ into above equation will obtain
\[    \frac{a^4-8a^2-8}{8a^2(1-a^2)^{\frac{7}{2}}}\frac{d^2\omega}{dt^2}+\frac{2a^8+13a^6+51a^4+16a^2+8}{8a^2(2a^2+1)(1-a^2)^{\frac{9}{2}}}\omega+\mu \frac{2a^2+1}{a(1-a^2)^{\frac{5}{2}}}=0.\]
It can be turned into 
\[\frac{d^2\omega}{dt^2}+\frac{2a^8+13a^6+51a^4+16a^2+8}{(2a^2+1)(1-a^2)(a^4-8a^2-8)}\omega+\mu \frac{8a(2a^2+1)(1-a^2)}{a^4-8a^2-8}=0.\]

When $a\in (0,1)$, the coefficient $b= \frac{2a^8+13a^6+51a^4+16a^2+8}{(2a^2+1)(1-a^2)(a^4-8a^2-8)}<0$.
If we denote $c=\frac{8a(2a^2+1)(1-a^2)}{a^4-8a^2-8}$, according to the theory of ODE, the solutions of Jacobi equation can be written as \begin{equation*} w=c_1\theta_1(t) + c_2\theta_2(t)+\mu\theta_3(t) \end{equation*}
where $c_1,c_2$ are abitrary constants and
 \begin{equation*} \begin{cases} 
 \theta_1(t) = \sinh \sqrt{-b}t, \\
 \theta_2(t) = \cosh \sqrt{-b}t, \\
  \theta_3(t) =- \frac{c}{b}. \end{cases}\end{equation*}
 
We know that if $t_0, t_1$ are conjugate points along the isoperimetric extremal circles if and only if we can find constants $c_1,c_2,\mu$ such that
\begin{equation}\label{eqn36}\begin{cases}
w(t_0)=c_1\theta_1(t_0) + c_2\theta_2(t_0)+\mu\theta_3(t_0) =0, \\
w(t_1)=c_1\theta_1(t_0) + c_2\theta_2(t_1)+\mu\theta_3(t_1) =0, \\
\int_{t_0}^{t_1} w \mathrm{d}t =c_1 \int_{t_0}^{t_1} U \theta_1\mathrm{d}t + c_2\int_{t_0}^{t_1} U \theta_2 \mathrm{d}t  +\mu \int_{t_0}^{t_1} U \theta_3 \mathrm{d}t  =0.
\end{cases} \end{equation}
Therefore we have \begin{equation}\label{eqn37}D(t_0,t_1) =\begin{vmatrix}
\theta_1(t_0) & \theta_2(t_0) & \theta_3(t_0) \\
\theta_1(t_1) & \theta_2(t_1) & \theta_3(t_1) \\
\int_{t_0}^{t_1} U \theta_1\mathrm{d}t  &\int_{t_0}^{t_1} U \theta_2\mathrm{d}t  & \int_{t_0}^{t_1} U \theta_3\mathrm{d}t 
\end{vmatrix} =0 .\end{equation}
 
It is easy to get the expression of $D(t_0, t_1)$:
\[D(t_0, t_1)=\frac{4cU}{(-b)^{3/2}}\sinh \frac{\sqrt{-b}(t_1-t_0)}{2}[\sinh \frac{\sqrt{-b}(t_1-t_0)}{2} - \frac{\sqrt{-b}(t_1-t_0)}{2}\cosh \frac{\sqrt{-b}(t_1-t_0)}{2} ]\]
where $U=\frac{2a^2+1}{a(1-a^2)^{\frac{5}{2}}}.$
Obviously $D(t_0,t_1)>0$ if $t_1>t_0$, which means there are no conjugate points along $c_0$. Summarising above discussion leads to the following theorem.

\begin{theorem}
Along the isoperimetric extremal circles $c_0=(a\cos(t), a\sin(t))$ of the integral $J$ in $(F_B, B^2(1))$, there are no conjugate points. 
\end{theorem}
 
\section{\textbf{The conclusion}}
\begin{theorem} 
Let $c_0$ be the circle centered at the origin in $B^2(1)$, the $c_0$ is a proper strong maximum of the isoperimetric problem, when using the Holmes-Thompson area. 
\end{theorem}
\begin{proof} At the critical curve $c_0=(a\cos(t), a\sin(t))$,  it can be verified that 
  \[\sum_{i,j=1}^2 h_{\dot{x}_i \dot{x}_j}y^iy^j=\frac{\lambda_0 (2a^2+1)}{a(1-a^2)^{\frac{5}{2}}}(\cos(t)y_1+\sin (t)y_2)^2.\]
 For $\lambda_0<0$, the inequality $\sum_{i,j=1}^2 h_{\dot{x}_i \dot{x}_j}y_iy_j<0$ holds for all $y\neq k \dot{c}_0(t)$. By the Theorem 2.3 and combining the results we have proved in previous section will lead to the conclusion.  
\end{proof}
 
Actually, we conjecture $c_0$ achieves the global maximum of the isoperimetric problem in $(F_B, B^2(1))$, when using the Holmes-Thompson area. 
\begin{conjecture}
Let $c_0$ be the circle centered at the origin in $(F_B, B^2(1))$, then $c_0$ encloses the maximal Holmes-Thompson area among all the simple closed curves which are smooth and have the fixed length. 
\end{conjecture}


{\small DEPARTMENT OF MATHEMATICS, EAST CHINA NORMAL UNIVERSITY }

{\small SHANGHAI 200062, CHINA}

{\small E-mail address: 51160601143@st.ecnu.edu.cn}

\vspace{6pt}

{\small DEPARTMENT OF MATHEMATICS, EAST CHINA NORMAL UNIVERSITY }

{\small SHANGHAI 200062, CHINA}

{\small E-mail address: lfzhou@math.ecnu.edu.cn}

\end{document}